\documentclass[11pt]{amsart}

\usepackage{amsfonts, amssymb, amsmath, graphicx}
\usepackage[all, knot]{xy}

\newcommand{\Z}{\mathbb{Z}}
\newcommand{\R}{\mathbb{R}}
\newcommand{\brak}[1]{\langle #1\rangle}

\def\calR{{\mathcal R}}
\def\calC{{\mathcal C}}
\def\calI{{\mathcal I}}

\newtheorem{theorem}{Theorem}

\newtheorem{lemma}{Lemma}
\newtheorem{proposition}[theorem]{Proposition}
\newtheorem{definition}[theorem]{Definition}

\newtheorem{remark}[theorem]{Remark}

\def\co{\colon\thinspace}

\begin{document}
\title{Invariants for trivalent tangles and handlebody-tangles}

\author{Carmen Caprau}
\address{Department of Mathematics, California State University, Fresno, CA 93740, USA}
\email{ccaprau@csufresno.edu}

\thanks{The author was partially supported by Simons Foundation grant $\#355640$}

\begin{abstract}
An enhanced trivalent tangle is a trivalent tangle with some of its edges labeled. We use enhanced trivalent tangles and classical knot theory to provide a recipe for constructing invariants for trivalent tangles, and in particular, for knotted trivalent graphs. Our method also yields invariants of, what we refer to as, enhanced handlebody-tangles and enhanced handlebody-links.
\end{abstract}

\maketitle

\section{Introduction} \label{intro}

A trivalent graph is a finite graph whose vertices have valency three, and a uni-trivalent graph is a finite graph whose vertices have valency three or one. In this paper, trivalent graphs and uni-trivalent graphs are not oriented, and may contain circle components.
A \textit{knotted trivalent graph} is a trivalent graph embedded in three-dimensional space, and a \textit{trivalent tangle} is a uni-trivalent graph embedded in $\R^2 \times [0,1]$ such that all of its univalent vertices belong to $\R^2 \times \{0\}$ and $\R^2 \times \{1\}$. We call a univalent vertex an \textit{endpoint} of the trivalent tangle. We note that a trivalent tangle with no endpoints is a knotted trivalent graph. Therefore, any statement that holds for trivalent tangles holds automatically for knotted trivalent graphs.

Two knotted trivalent graphs are called \textit{equivalent} (or \textit{ambient isotopic}) if there is an isotopy of $\R^3$ taking one onto the other. Moreover, two trivalent tangles are called equivalent if one can be transformed into the other by an isotopy of $\R^2 \times [0,1]$ fixed on the boundary.
It is well-known that two knotted trivalent graphs (or trivalent tangles) are equivalent if and only if their diagrams are related by a finite sequence of the moves R1 -- R5 depicted in Figure~\ref{fig:graph-moves} (see~\cite{Ka} for more details).

\begin{figure}[ht] 
\[\raisebox{-13pt}{\includegraphics[height=0.4in]{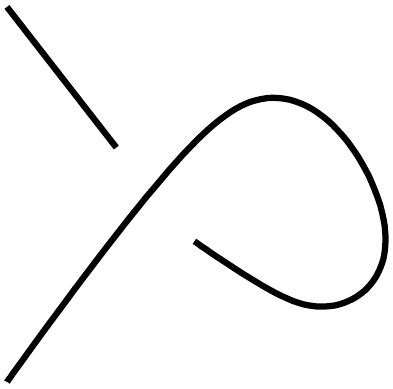}}\,\,  \stackrel{\text{R1}}{\longleftrightarrow} \,\, \raisebox{-13pt}{\includegraphics[height=0.4in]{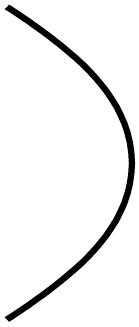}} \,\, \stackrel{\text{R1}}{\longleftrightarrow} \,\,   \reflectbox{\raisebox{17pt}{\includegraphics[height=0.4in, angle = 180]{poskink}}} \,\,\qquad\,\, 
 \raisebox{-13pt}{\includegraphics[height=0.4in]{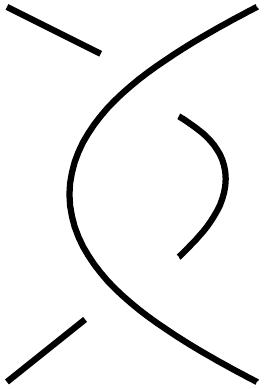}} \,\, \stackrel{\text{R2}}{\longleftrightarrow} \,\, \raisebox{-13pt}{\includegraphics[height=0.4in]{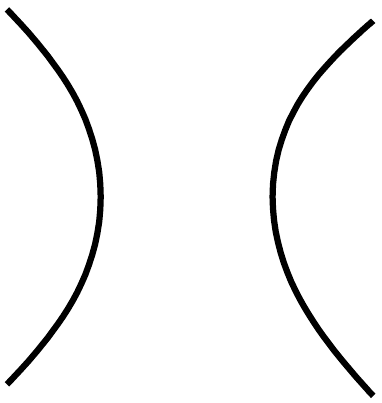}} \,\, \qquad \,\, 
\raisebox{-13pt}{\includegraphics[height=0.4in]{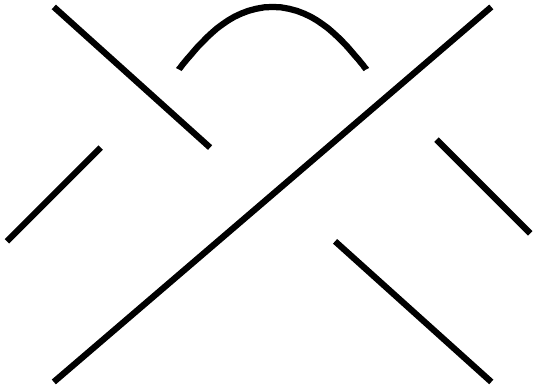}} \,\, \stackrel{\text{R3}}{\longleftrightarrow} \,\,   \raisebox{-13pt}{\includegraphics[height=0.4in]{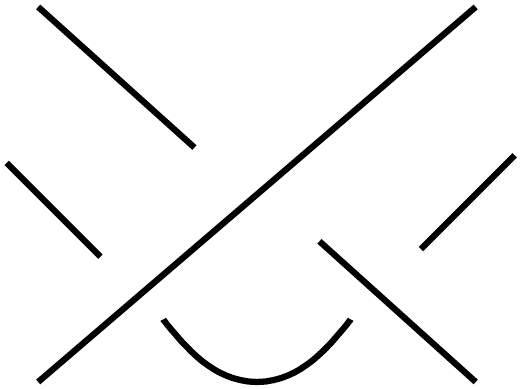}}\]
\[  \raisebox{-13pt}{\includegraphics[height=0.45in]{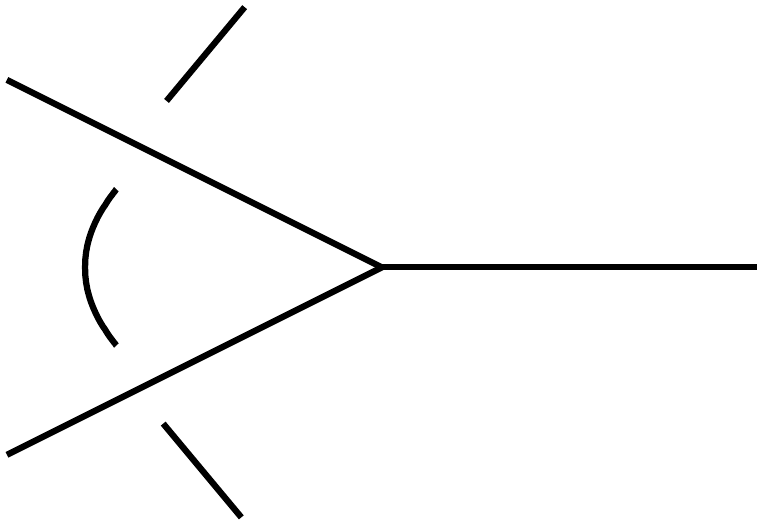}}\,\, \stackrel{\text{R4}}{\longleftrightarrow} \,\, \raisebox{-13pt}{\includegraphics[height=0.45in]{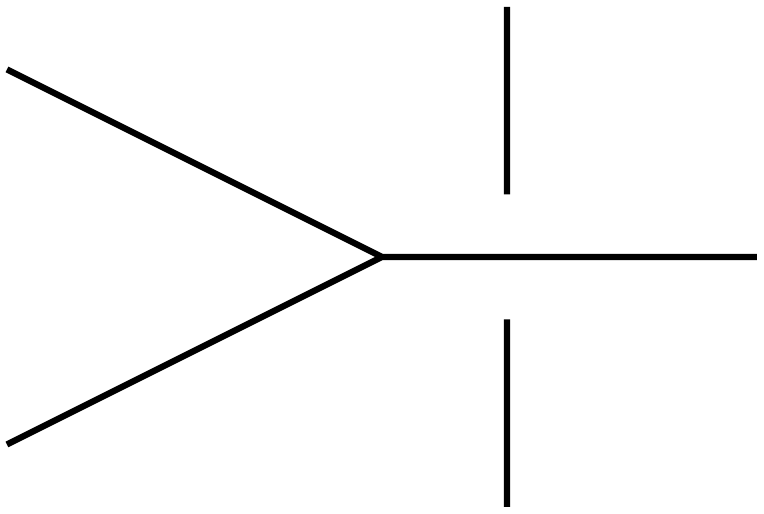}} \,\, \qquad \,\,  \raisebox{-13pt}{\includegraphics[height=0.45in]{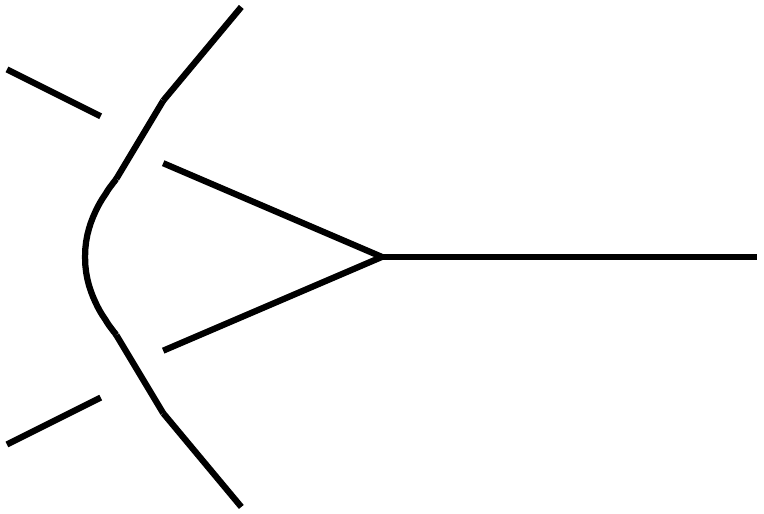}} \,\, \stackrel{\text{R4}}{\longleftrightarrow}\,\,  \raisebox{-13pt}{\includegraphics[height=0.45in]{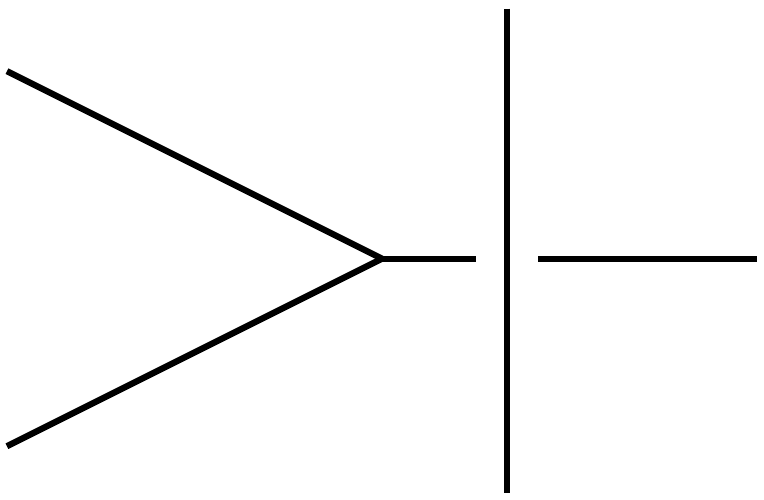}} \]
\[ \raisebox{-13pt}{\includegraphics[height=0.4in]{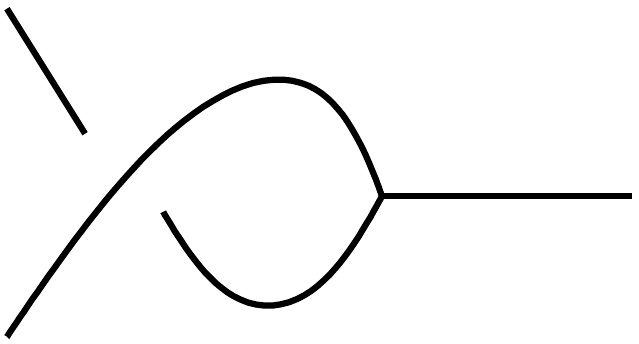}}\,\,\stackrel{\text{R5}}{ \longleftrightarrow} \,\, \raisebox{-13pt}{\includegraphics[height=0.4in]{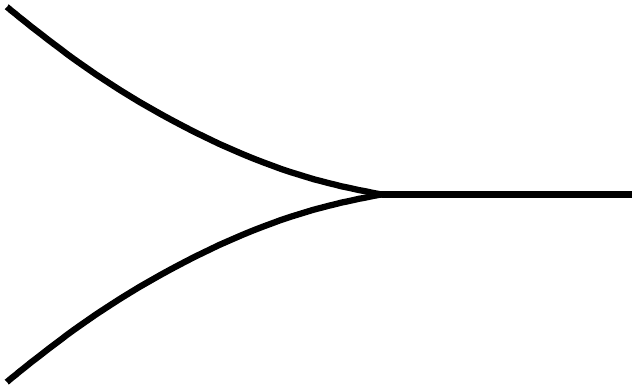}} \,\, \stackrel{\text{R5}}{\longleftrightarrow} \,\,  \raisebox{-13pt}{\includegraphics[height=0.4in]{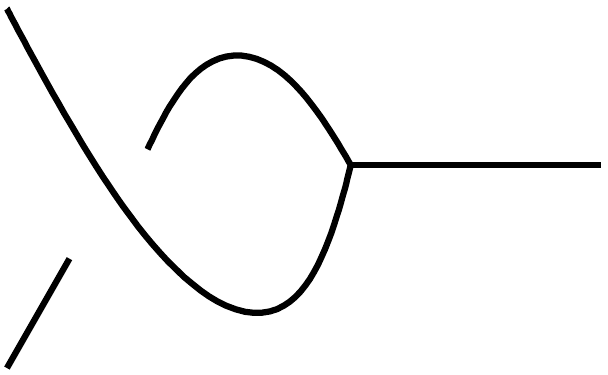}}\]
\caption{Moves for knotted trivalent graph diagrams} \label{fig:graph-moves}
\end{figure}

Two trivalent tangles with the same endpoints are called \textit{neighborhood equivalent} if there is an isotopy of $\R^2 \times [0,1]$ (which fixes its boundary) taking a regular neighborhood of one onto a regular neighborhood of the other. An \textit{IH-move} is a local change of a trivalent tangle, as shown in Figure~\ref{fig:IH-move}, applied in a disk embedded in the interior of $\R^2 \times [0,1]$. Two trivalent tangles with the same endpoints are called \textit{IH-equivalent} if they are related by a finite sequence of IH-moves and isotopies of $\R^2 \times [0,1]$ fixed on the boundary. Ishii~\cite{I1} showed that two trivalent tangles with the same endpoints (and respectively, two knotted trivalent graphs) are neighborhood equivalent if and only if they are IH-equivalent.  

\begin{figure}[ht]\
\[ \raisebox{-12pt}{\includegraphics[height=0.4in]{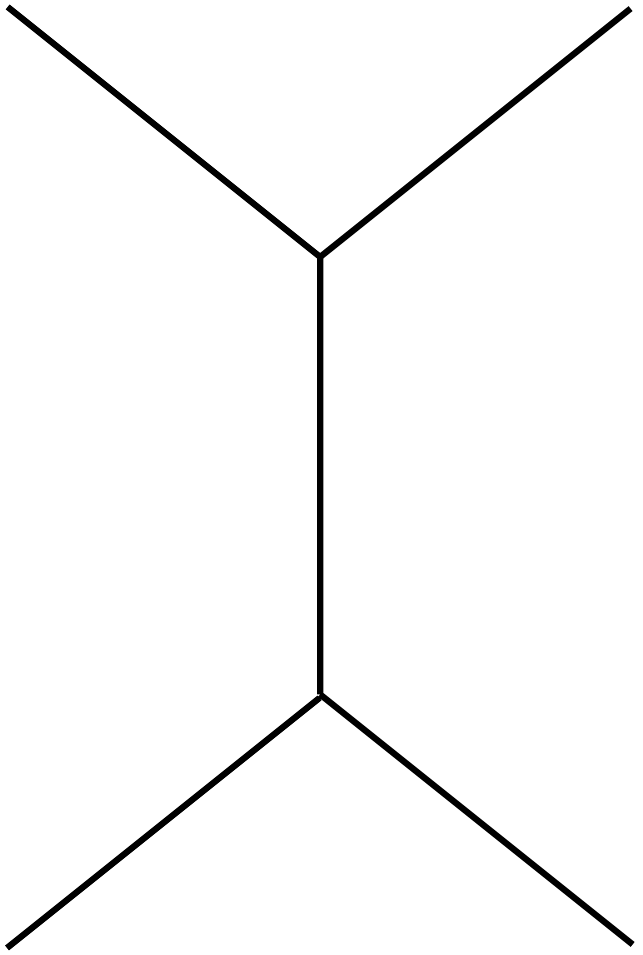}}\,\, \longleftrightarrow \,\, \raisebox{-12pt}{\includegraphics[height=0.4in]{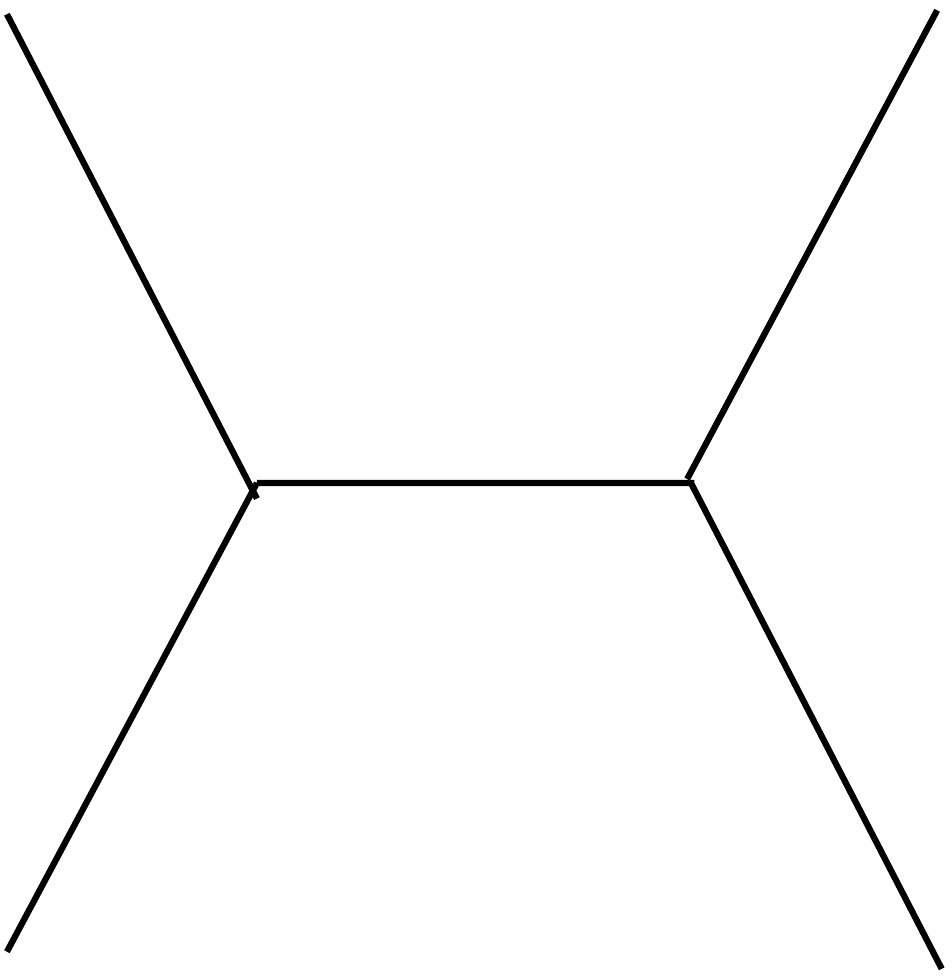}} \]
\caption{IH-move} \label{fig:IH-move}
\end{figure}

It follows that two trivalent tangles (respectively, knotted trivalent graphs) are neighborhood equivalent if and only if their diagrams are related by a finite sequence of the moves R1 -- R5 together with IH-moves (here regarded as moves in the plane). 

 A \textit{handlebody-tangle} is a disjoint union of handlebodies embedded in the three-ball $\R^2 \times [0,1]$ such that the intersection of the handlebodies with $\R^2 \times \{0\}$ and $\R^2 \times \{1\}$ consists of disks, called \textit{end disks} of the handlebody-tangle. A handlebody-tangle with no end disks is called a \textit{handlebody-link}. Two handlebody-tangles are called \textit{equivalent} if there exists an orientation-preserving homeomorphism of $\R^2 \times [0,1]$ into itself taking one onto the other, and which is the identity map on the boundary.

 Any handlebody-tangle  is a regular neighborhood of some trivalent tangle, and therefore, there is a one-to-one correspondence between the set of handlebody-tangles and the set of the neighborhood equivalence classes of trivalent tangles.  When a handlebody-tangle $H$ is a regular neighborhood of some trivalent tangle $T$ such that each end disk of $H$ contains exactly one endpoint of $T$, we say that $T$ is a \textit{spine} of $H$ (or that $H$ is \textit{represented} by $T$). Figure~\ref{fig:handlebody-tangle} shows a handlebody-tangle and two spines that represent it.

\begin{figure} [ht]
\[ \raisebox{-13pt}{\includegraphics[height=0.8in]{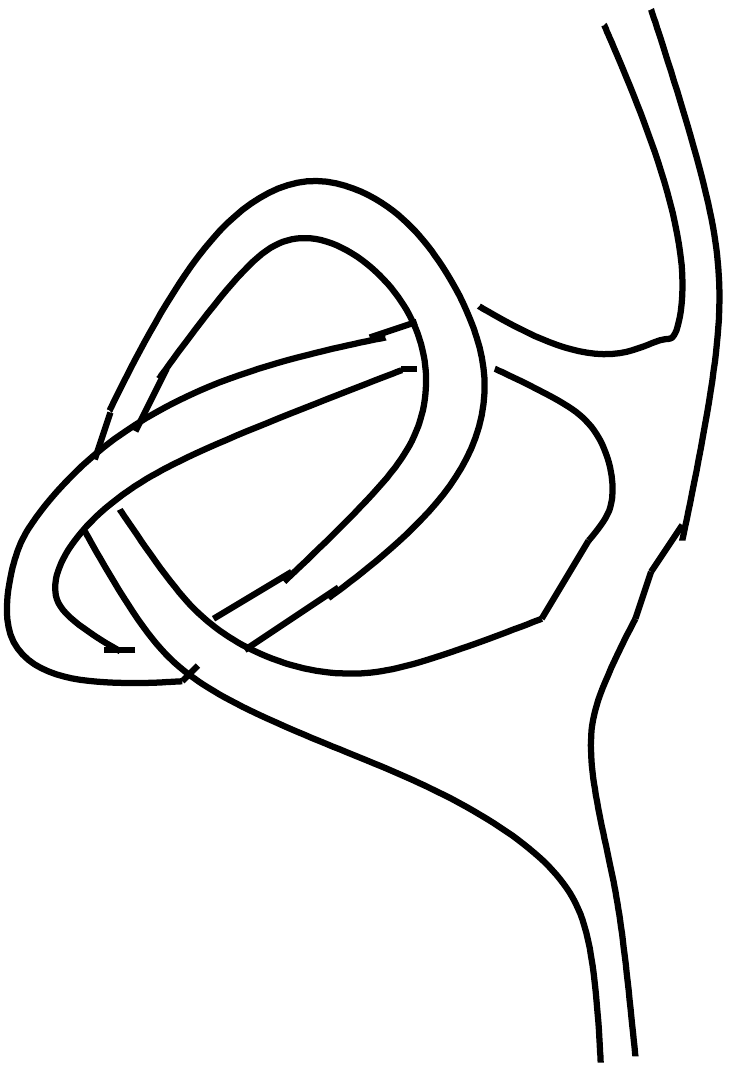}} \qquad \raisebox{-13pt}{\includegraphics[height=0.8in]{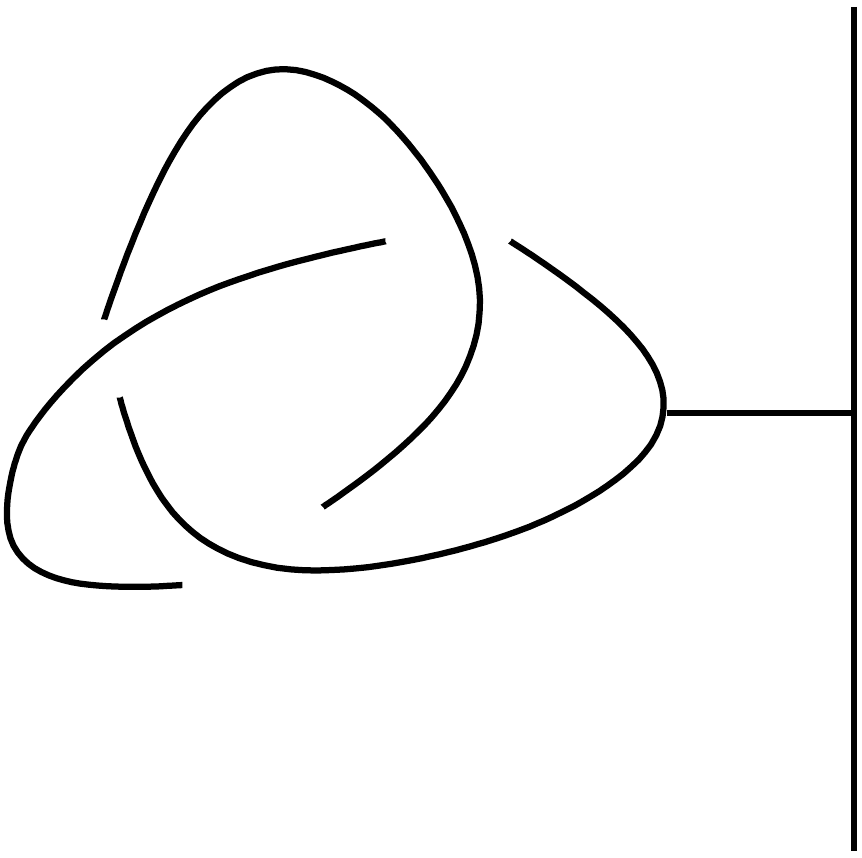}} \qquad \raisebox{-13pt}{\includegraphics[height=0.8in]{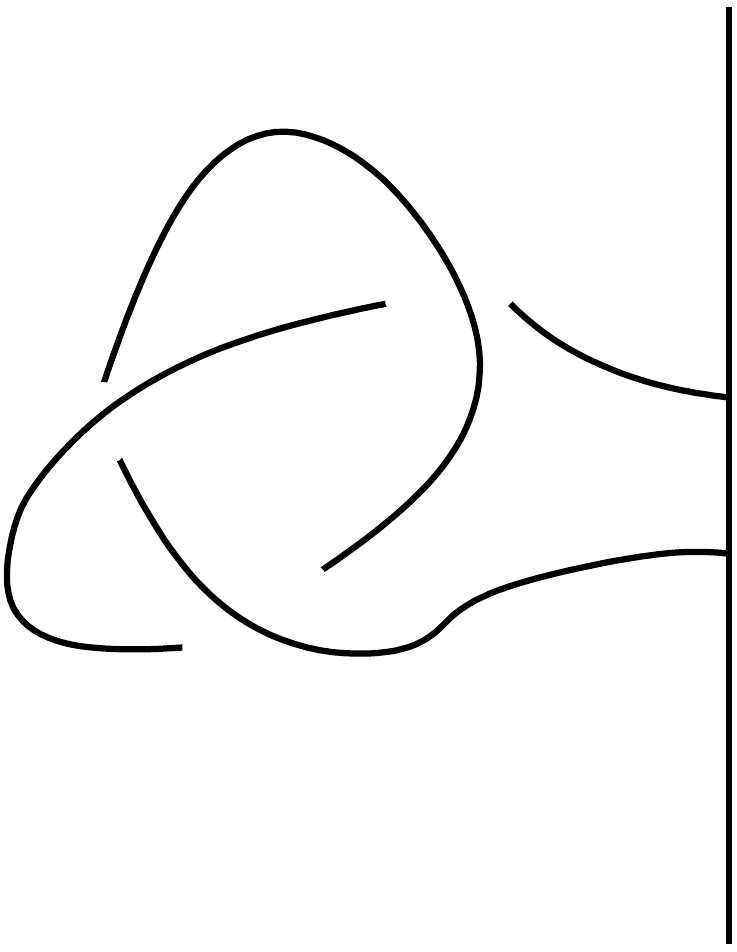}} \]
\caption{A handlebody-tangle and two spines of it} \label{fig:handlebody-tangle}
\end{figure}

The above statements imply that two trivalent tangles with the same endpoints represent equivalent handlebody-tangles if and only if their diagrams are related by a finite sequence of the moves R1 -- R5 together with IH-moves. For more details we refer the reader to~\cite{I1} (see also~\cite{I2}).  

Therefore, we can study handlebody-tangles  through diagrams of trivalent tangles.  To obtain an invariant for a handlebody-tangle $H$ represented by a trivalent tangle $T$, it suffices to construct an invariant of the IH-equivalence class of $T$; that is, one needs to associate to a diagram of $T$ some quantity which is invariant under the moves R1 -- R5, as well as under the IH-move.
Similarly, one can study handlebody-links through diagrams of knotted trivalent graphs up to the moves R1--R5 and IH-move.

In this paper, we introduce the notions of \textit{enhanced trivalent tangles} and \textit{enhanced handlebody-tangles}. An enhanced trivalent tangle is a trivalent tangle with an edge set on which IH-moves can be applied. An enhanced handlebody-tangle is a handlebody-tangle represented by an enhanced trivalent tangle. We use enhanced trivalent tangles and combinatorial knot theory to provide a general recipe for constructing invariants for trivalent tangles (and, in particular, for knotted trivalent graphs). We also construct numerical invariants for trivalent tangles; these invariants depend on the definition of the Kauffman bracket of classical knots and links. The recipe provided in this paper also yields invariants of enhanced handlebody-tangles.

The paper is organized as follows: In Section~\ref{ssec:enhanced graphs} we introduce the notions of enhanced trivalent tangles and that of IH-equivalence classes of enhanced trivalent tangles. Then we explain that there is a one-to-one correspondence between the set of IH-equivalence classes of enhanced trivalent tangles and the set of ambient isotopy classes of $4$-valent tangles (see Lemma~\ref{lemma}).
We use this statement to provide a recipe for constructing invariants of the IH-equivalence class of an enhanced trivalent tangle $(G, \rho)$ with diagram $D_{\rho}$ via a collection $\calC(\overline{D}_{\rho})$ of knot theoretic tangle diagrams associated with $\overline{D}_{\rho}$; here $\overline{D}_{\rho}$ is the $4$-valent tangle diagram obtained from $D_{\rho}$ by contracting its thick edges (see Proposition~\ref{prop:inv-enhanced}). In Section~\ref{ssec:inv trivalent tangles} we show how one can use 3-move invariants of knot theoretic tangles to finally arrive at invariants of IH-equivalence classes of enhanced trivalent tangles, and of enhanced handlebody-tangles. Therefore, it remains to find 3-move invariants for classical tangles. Given an $(m, n)$-tangle $T$ with diagram $D$, in Section~\ref{sec:inv tangles} we use skein modules and basic linear algebra concepts to define a polynomial $P(D) \in \mathbb{Z}[q, q^{-1}]$ in terms of the skein class $\brak{D}$ of $D$ (see Definition~\ref{def:P(D)}). It turns out that $P(D)$ is equal to the unnormalized Kauffman bracket of the knot or link obtained by taking the plat closure of the tangle $T \otimes \overline{T}$, where $\overline{T}$ is the mirror image of $T$. In Theorem~\ref{main result} we prove that, for each $k \in \{1, 5, 7, 11, 13, 17, 19, 23\}$, the complex number $P(D) |_{q = e^{\frac{k \pi i}{12}}}$ is a 3-move invariant for the $(m, n)$-tangle $T$.


\section{Constructing invariants for trivalent tangles}  \label{sec:trivalent tangles}

\subsection{Enhanced trivalent tangles}\label{ssec:enhanced graphs}

In this paper, handlebody-tangles have an even number of end disks, and trivalent tangles have an even number of endpoints (univalent vertices). Equivalently, a trivalent tangle contains an even number of trivalent vertices. Recall that any knotted trivalent graph contains an even number of trivalent vertices. 

Let $G$ be a trivalent tangle. We call an edge of $G$ joining two trivalent vertices an \textit{internal edge}, and an edge incident to an endpoint of $G$ an \textit{external edge}. Let $\rho$ be a map from the set of edges of $G$ to the set $\{1,2 \}$ such that $\rho(e_1) + \rho(e_2) + \rho(e_3) = 4$ for edges $e_1, e_2, e_3$ incident to a trivalent vertex, under the restriction that external edges and cycles or loops may be assigned only the value $1$. Denote by $\calR(G)$ the set of all such maps, and call the pair $(G, \rho)$, for some $\rho \in \calR(G)$, an \textit{enhanced trivalent tangle} associated to $G$. We represent an edge $e$ for which $\rho(e) = 2$ by a `thick' edge in a diagram of $(G, \rho)$. We note that for an enhanced trivalent tangle, there is at most one thick edge joining a pair of adjacent trivalent vertices, and that no external edge is a thick edge. An \textit{enhanced knotted trivalent graph} is an enhanced trivalent tangle with no endpoints.

We say that two enhanced trivalent tangles $(G_1, \rho_1)$ and $(G_2, \rho_2)$ with the same endpoints are \textit{equivalent} (or \textit{ambient isotopic}) if there exists an orientation-preserving homeomorphism $f \co \R^2 \times [0,1] \to \R^2 \times [0,1]$ (fixing the boundary) such that $f(G_1) = f(G_2)$ and $f(E_{\rho_1}) = f(E_{\rho_2})$, where $E_{\rho_1}$ and $E_{\rho_2}$ are the sets of thick edges in $(G_1, \rho_1)$ and $(G_2, \rho_2)$, respectively.  An \textit{IH-move on enhanced trivalent tangles} is an IH-move which replaces thick edges with thick edges. 
We say that two enhanced trivalent tangles with the same endpoints are \textit{$IH$-equivalent} if they are related by a finite sequence of IH-moves on thick edges and isotopies of $\R^2 \times [0,1]$ fixed on the boundary. These definitions extend to enhanced knotted trivalent graphs.

An \textit{enhanced handlebody-tangle} (respectively, \textit{enhanced handlebody-link}) is the IH-equivalence class of an enhanced trivalent tangle (respectively, enhanced knotted trivalent graph). It follows that in order to construct an invariant for an enhanced handlebody-tangle $H$ represented by an enhanced trivalent tangle $(G, \rho)$, it suffices to construct an invariant of the IH-equivalence class of $(G, \rho)$. 

For each enhanced  trivalent tangle $(G, \rho)$, there exists an associated $4$-valent tangle $\overline{G}_{\rho}$ obtained by contracting each thick edge $\textbf{e}$ in $(G, \rho)$. A \textit{$4$-valent tangle} is a uni-four-valent graph embedded in $B^3$, whose intersection with $\partial B^3$ consists of its univalent vertices (or endpoints). A \textit{contraction move} is a local change as depicted in Figure~\ref{fig:contraction move}, where the replacement is applied in a disk embedded in the interior of $B^3$.
\begin{figure}[ht]
\[  \raisebox{-7pt}{\includegraphics[height=0.3in]{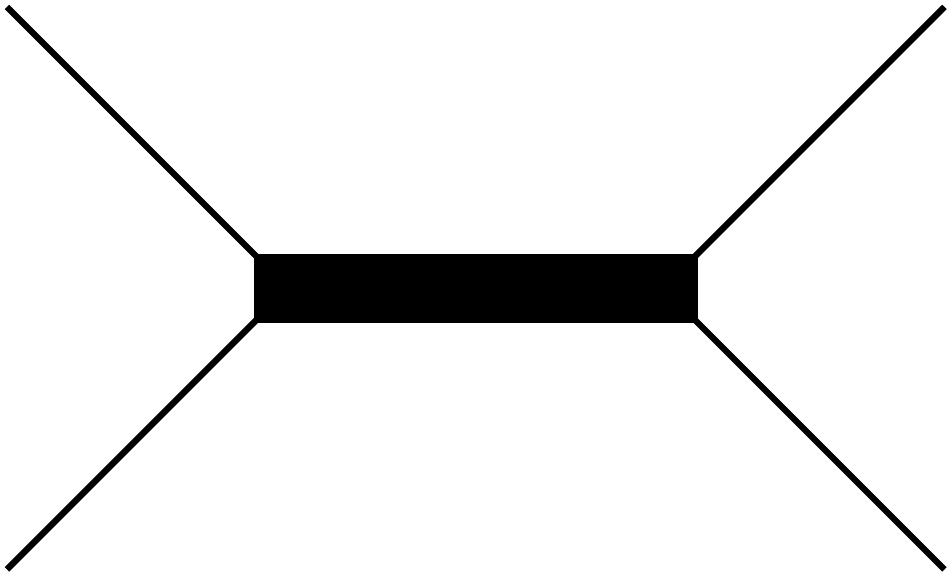}} \,\, \longleftrightarrow \,\,  \raisebox{-7pt}{\includegraphics[height=0.3in]{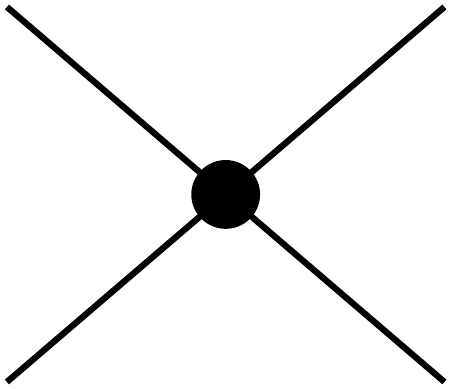}} \]
\caption{Contraction move} \label{fig:contraction move}
\end{figure}

Recall that two knotted $4$-valent graphs (or two $4$-valent tangles with the same endpoints) are \textit{ambient isotopic} if and only if their diagrams are related by a finite sequence of the Reidemeister moves R1 -- R3 and the moves N4 -- N5 given in Figure~\ref{fig:isotopies 4-valent graphs} below (see~\cite{Ka}).
\begin{figure}[ht]
\[  \raisebox{-12pt}{\includegraphics[height=0.4in]{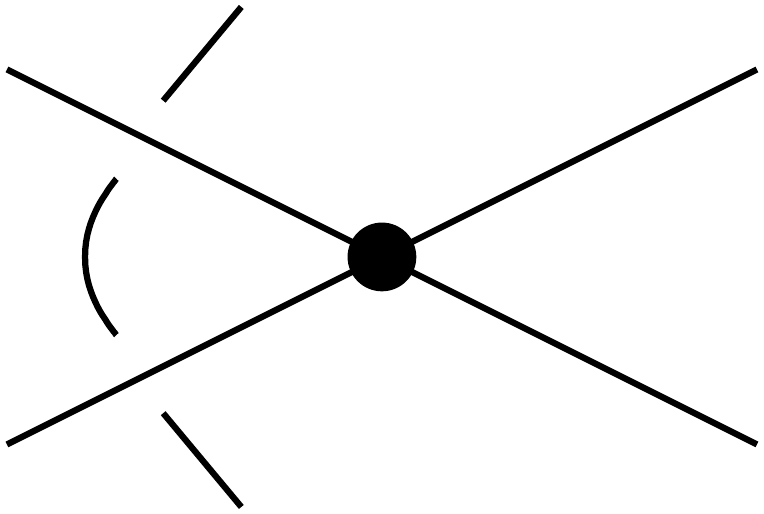}}\,\, \stackrel{\text{N4}}{\longleftrightarrow} \,\, \raisebox{-12pt}{\includegraphics[height=0.4in]{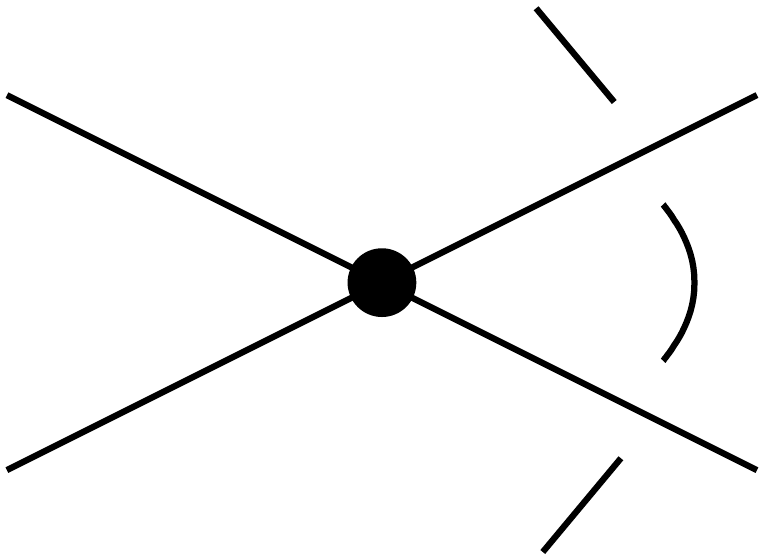}} \,\, \qquad \,\,  \raisebox{-12pt}{\includegraphics[height=0.4in]{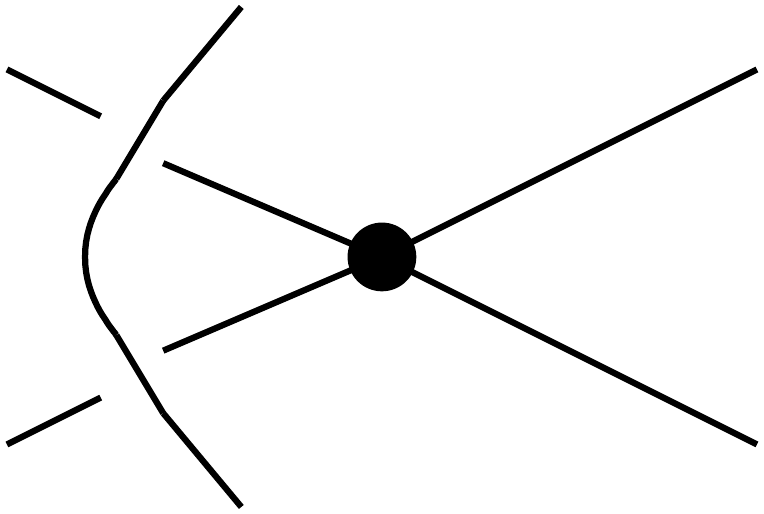}} \,\, \stackrel{\text{N4}}{\longleftrightarrow}\,\,  \raisebox{-12pt}{\includegraphics[height=0.4in]{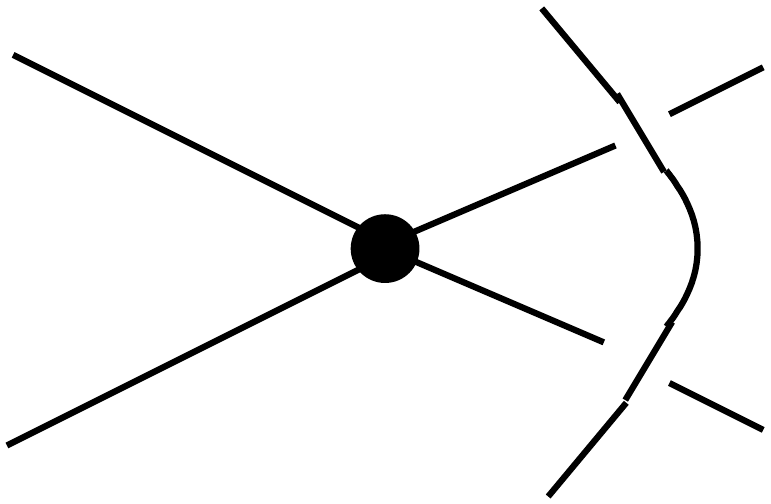}} \]
\[ \raisebox{-10pt}{\includegraphics[height=0.35in]{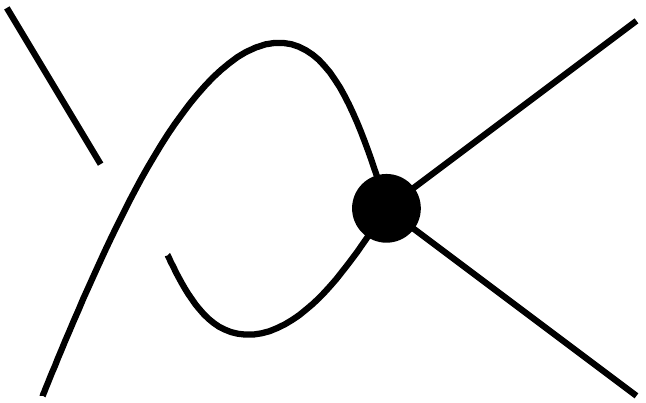}}\,\,\stackrel{\text{N5}}{ \longleftrightarrow} \,\, \raisebox{-10pt}{\includegraphics[height=0.35in]{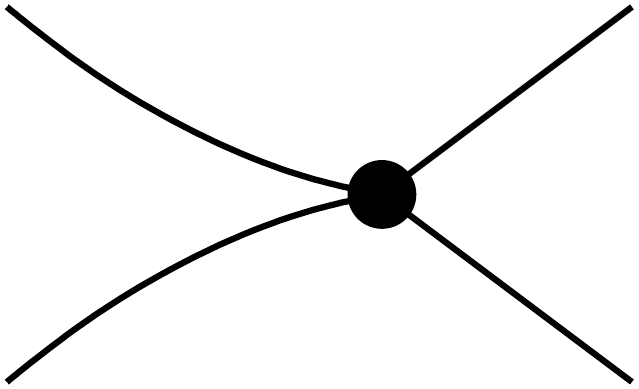}} \,\, \stackrel{\text{N5}}{\longleftrightarrow} \,\,  \raisebox{-10pt}{\includegraphics[height=0.35in]{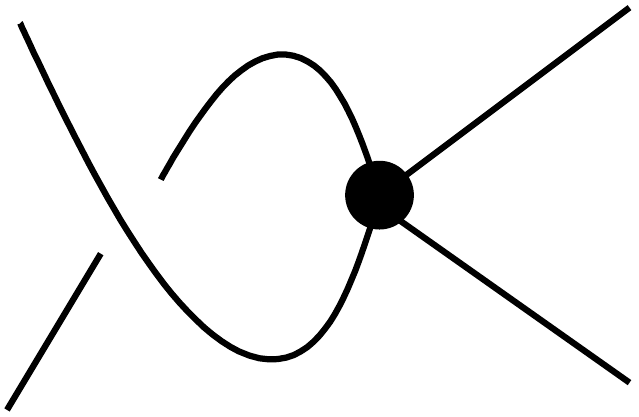}}\]
\caption{Moves for $4$-valent graph diagrams}\label{fig:isotopies 4-valent graphs}
\end{figure}

Let $D_{\rho}$ be a diagram of an enhanced  trivalent tangle $(G, \rho)$, and denote by $\overline{D}_{\rho}$ a diagram of the associated $4$-valent tangle $\overline{G}_{\rho}$. If $D_{\rho}$ and $D'_{\rho}$ are diagrams of an enhanced trivalent tangle $(G, \rho)$, then there are diagrams $\overline{D}_{\rho}$ and $\overline{D'}_{\rho}$ representing the $4$-valent tangle $\overline{G}_{\rho}$ obtained from $(G, \rho)$ by applying the contraction move given in Figure~\ref{fig:contraction move}. In particular, diagrams $\overline{D}_{\rho}$ and $\overline{D'}_{\rho}$ can be obtained from $D_{\rho}$ and $D'_{\rho}$, respectively,  by contracting their thick edges. 
Therefore, we can study an enhanced trivalent tangle $(G, \rho)$ through diagrams $\overline{D}_{\rho}$ of $4$-valent tangles.

A few words are needed here, as a thick edge in the diagram $D_{\rho}$ might cross under or over (at least) an edge, making the contraction move for diagrams of trivalent graphs/tangles somewhat ambiguous. Below we exemplify the case of a thick edge crossing over a `thin' edge, where we see that the two $4$-valent diagrams on the right are the same, up to the move N4. 
 \[ \raisebox{-9pt}{\includegraphics[height=0.35in]{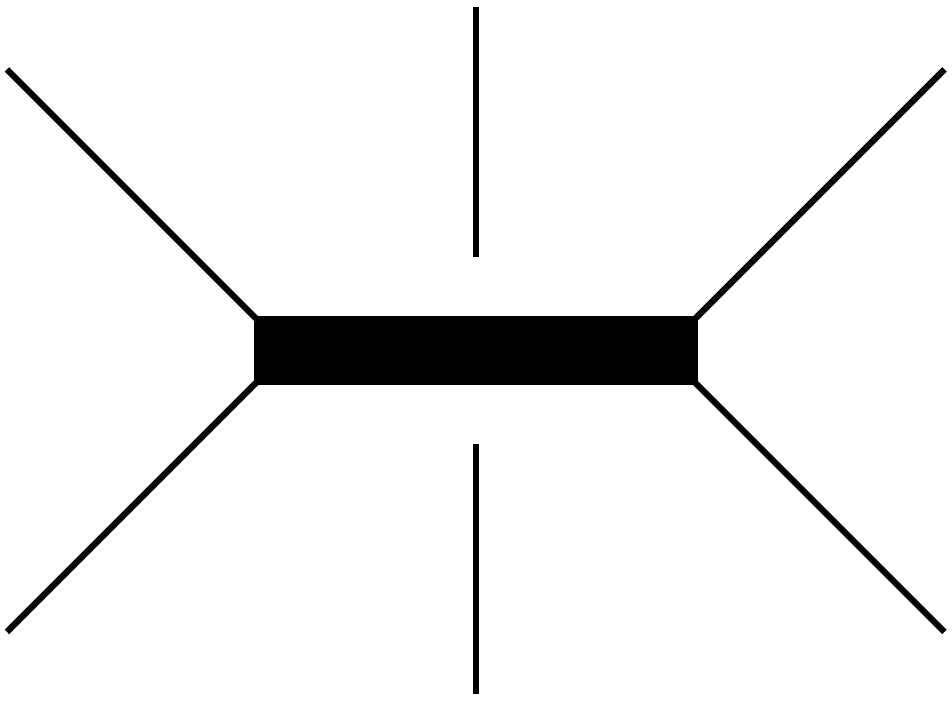}} \,\, \longrightarrow \,\,  \raisebox{-9pt}{\includegraphics[height=0.35in]{N4-left1}} \,\,\, \mbox{or} \,\,\, \raisebox{-9pt}{\includegraphics[height=0.35in]{N4-right1}} \]
The case in which a thick edge crosses under and/or over a few edges, some of which may be thick edges, are also unambiguous up to the move N4. Since we will be working with $4$-valent tangle diagrams up to the moves N4 -- N5 and Reidemeister moves R1 -- R3, we can assume that a trivalent tangle diagram $D_{\rho}$ does not contain crossings involving thick edges (except for self intersection of thick edges).

The following statement follows from the above discussion.

\begin{lemma} \label{lemma}
There is a one-to-one correspondence between the sets of ambient isotopy classes, as well as IH-equivalence classes, of enhanced trivalent tangles (or enhanced knotted  trivalent graphs) and that of $4$-valent tangles (or knotted $4$-valent graphs). 
\end{lemma}
The next step is to create a collection $\calC(\overline{D}_{\rho})$  of knot theoretic tangle diagrams obtained via the local replacements depicted in Figure~\ref{fig:replacements}.
\begin{figure}[ht]
\[    \raisebox{-7pt}{\includegraphics[height=0.3in]{cross}} \,\, \leadsto \,\, \left \{ \,\ \raisebox{-7pt}{\includegraphics[height=0.3in]{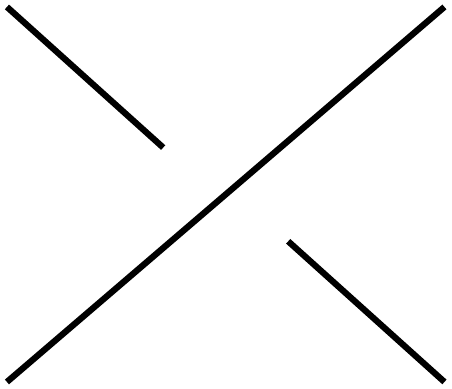}}, \,\,\,  \raisebox{-7pt}{\includegraphics[height=0.3in]{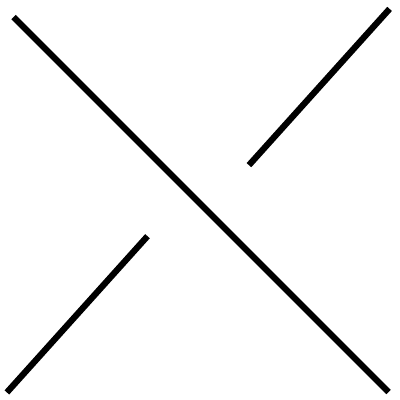}},\,\,\, \raisebox{-7pt}{\includegraphics[height=0.3in]{A-smoothing}}, \,\,\,  \raisebox{-7pt}{\includegraphics[height=0.3in, angle = 90]{A-smoothing}} \,\, \right \}\]
\caption{Local replacements}\label{fig:replacements}
\end{figure}

 Let $T_{-} = \raisebox{-5pt}{\includegraphics[height=0.2in]{poscrossing}}, \, T_{+} = \raisebox{-5pt}{\includegraphics[height=0.2in]{negcrossing}}, \, T_0 = \raisebox{-5pt}{\includegraphics[height=0.2in]{A-smoothing}}, \, T_{\infty} = \raisebox{-5pt}{\includegraphics[height=0.2in, angle = 90]{A-smoothing}}$ be the (2,2)-tangle diagrams depicted in Figure~\ref{fig:replacements}. 
 Let $\overline{D_{\rho}}$ be a diagram of a 4-valent tangle $\overline{G}_{\rho}$ with 4-valent vertex set $V(\overline{D_{\rho}})$. Let $f \co V(\overline{D_{\rho}}) \to \{ T_{-}, T_{+}, T_0, T_{\infty} \}$ be a function that assigns a member in $\{ T_{-}, T_{+}, T_0, T_{\infty} \}$ for each 4-valent vertex in $\overline{D}_{\rho}$ (or equivalently, for each thick edge of $D_{\rho}$). Observe that there are $4^n$ assignments of $\overline{D}_{\rho}$, where $n$ is the number of vertices in $\overline{D}_{\rho}$ (or equivalently, $n$ is the number of thick edges in $D_{\rho}$). Denote these assignments by $\{ f_1, f_2, \dots, f_{4^n} \}$. For each assignment $f_i $, denote by $(\overline{D_{\rho}}, f_i)$ the ordinary tangle diagram obtained from $\overline{D}_{\rho}$ by replacing the 4-valent vertices as prescribed by the assignment $f_i$. We call such a tangle diagram $(\overline{D}_{\rho}, f_i)$ a \textit{state} of $\overline{D}_{\rho}$.

Denote by $\calC(\overline{D}_{\rho}) : = \{ (\overline{D}_{\rho}, f_i)\,\, | \,\, 1 \leq i \leq 4^n  \}$ the collection of all states associated with $\overline{D}_{\rho}$. 
 
Two (ordinary) tangles are called \textit{3-equivalent} if their diagrams are related by a finite sequence of the three Reidemeister moves R1 -- R3, and the \textit{$3$-moves} below:
\[ \raisebox{-3pt}{\includegraphics[height=0.8in, angle = 90]{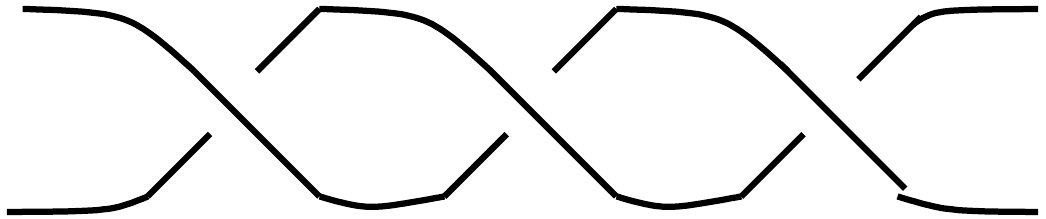}} \quad  \stackrel{+3-move}{\longleftrightarrow}  \quad \raisebox{-3pt}{\includegraphics[height=0.8in, angle=90]{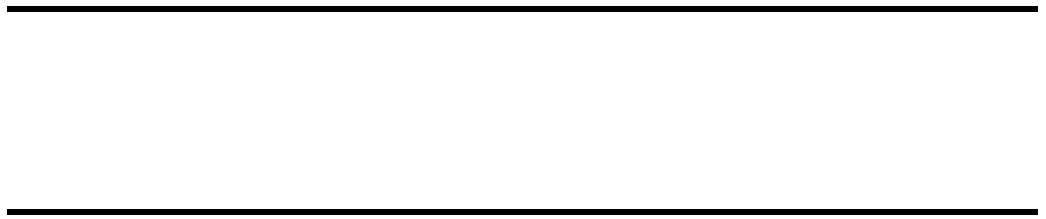}}\quad  \stackrel{-3-move}{\longleftrightarrow} \quad \reflectbox{\raisebox{-3pt}{\includegraphics[height=0.8in, angle = 90]{+3-move}}} \]

Two collections of tangles $S_1$ and $S_2$ are called \textit{3-equivalent} if every member of $S_1$ is 3-equivalent to some member of $S_2$. The following proposition is essentially Theorem 3.3 from~\cite{L-S}, thus we only sketch its proof.

\begin{proposition}\label{prop:3-move inv}
Let $\overline{D}_{\rho}$ and $\overline{D'}_{\rho}$ be two diagrams of a $4$-valent tangle $\overline{G}_{\rho}$ with $n$ $4$-valent vertices. Then there exists a permutation $\sigma$ of the set $\{1, \dots, 4^n \}$ such that the tangle $(\overline{D}_{\rho}, f_i)$ is $3$-equivalent to the tangle $(\overline{D'}_{\rho}, f_{\sigma(i)})$ for each $1 \leq i \leq 4^n$. In particular, the 3-equivalence class of the collection $\calC(\overline{D}_{\rho})$ is an ambient isotopy invariant of $\overline{G}_{\rho}$.
\end{proposition}

\begin{proof}
Without loss of generality, assume that $\overline{D'}_{\rho}$ is obtained from $\overline{D}_{\rho}$ by applying exactly one of the moves R1, R2, R3, N4 and N5.

Case I (moves R1 -- R3). It is obvious that the Reidemeister moves R1, R2 or R3 do not affect the local replacements at a $4$-valent vertex.

Case II (move N4). Suppose that $\overline{D}_{\rho}$ and $\overline{D'}_{\rho}$ are diagrams that are identical except in a small neighborhood where they differ by a move of type N4. Below we illustrate the effect of this move on local replacements at the involved vertex. 
\[ \overline{D}_{\rho} = \raisebox{-12pt}{\includegraphics[height=0.4in]{N4-left1}}\,\, \leadsto \left \{ \raisebox{-12pt}{\includegraphics[height=0.4in]{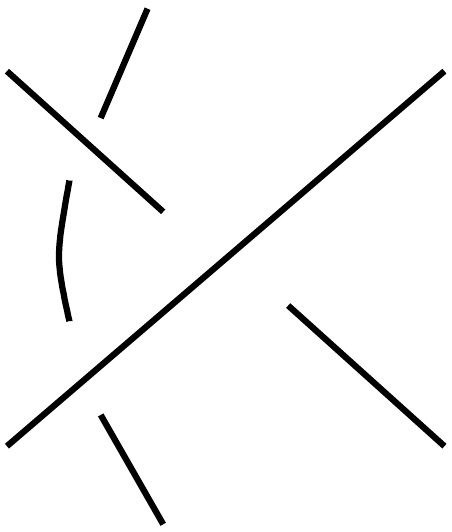}}, \,\, \  \raisebox{-12pt}{\includegraphics[height=0.4in]{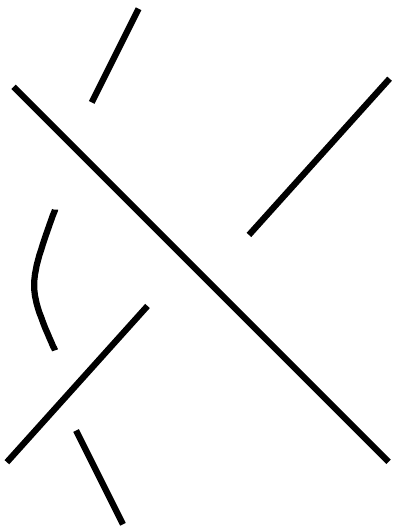}}, \,\, \, \raisebox{-12pt}{\includegraphics[height=0.35in]{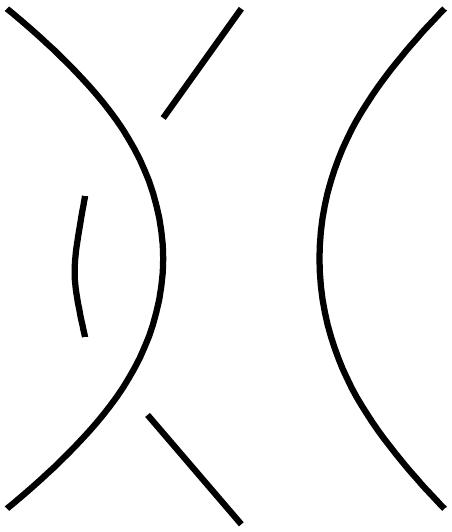}}, \,\, \, \raisebox{-12pt}{\includegraphics[height=0.37in]{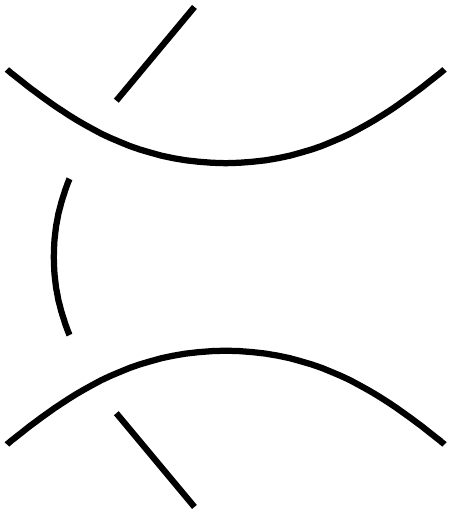}}  \right \}  \]
\[ \overline{D'}_{\rho} =  \raisebox{-12pt}{\includegraphics[height=0.4in]{N4-right1}}\,\, \, \leadsto \left \{ \raisebox{-12pt}{\includegraphics[height=0.4in]{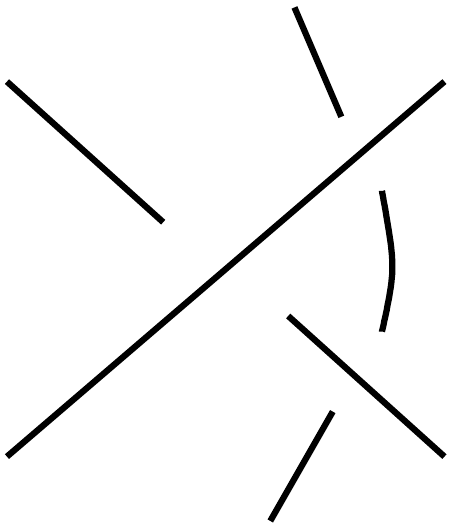}}, \,\,\,  \raisebox{-12pt}{\includegraphics[height=0.4in]{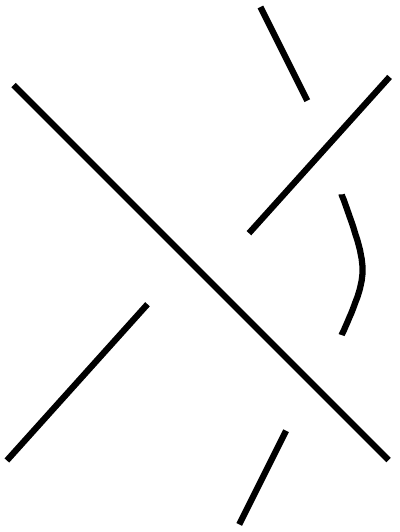}}, \,\, \, \raisebox{15pt}{\includegraphics[height=0.35in, angle = 180]{repcross2-5}}, \,\,\, \raisebox{-12pt}{\includegraphics[height=0.37in]{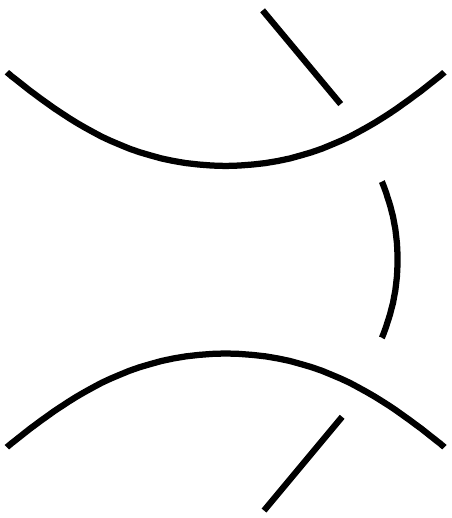}}  \right \}  \]
We see that the two collections above are ambient isotopic, and therefore, are 3-equivalent.

Case III (move N5). Suppose that $\overline{D}_{\rho}$ and $\overline{D'}_{\rho}$ are diagrams that are identical except in a small neighborhood where they differ by a move of type N5, as shown below:
\[ \overline{D}_{\rho} = \raisebox{-10pt}{\includegraphics[height=0.35in]{N5-left1}}\,\,\, \mbox{or} \,\,\, \raisebox{-10pt}{\includegraphics[height=0.35in]{N5-left2}}, \,\,\, \mbox{and} \,\,\, \overline{D'}_{\rho} = \raisebox{-10pt}{\includegraphics[height=0.35in]{N5-right}}  \]
The local replacements at the vertex involved in the move N5 are as follows:
\[  \raisebox{-10pt}{\includegraphics[height=0.35in]{N5-left1}} \leadsto \left\{\,\, \raisebox{-7pt}{\includegraphics[height=0.23in]{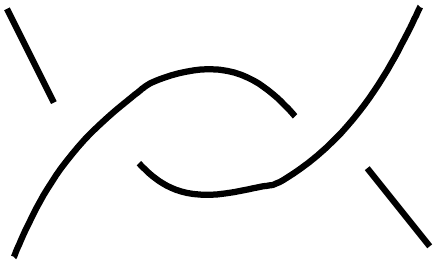}}\,,\, \reflectbox{\raisebox{11pt}{\includegraphics[height=0.38in, angle = 270]{reid2-1}}}\,,\, \raisebox{-7pt}{\includegraphics[height=0.25in]{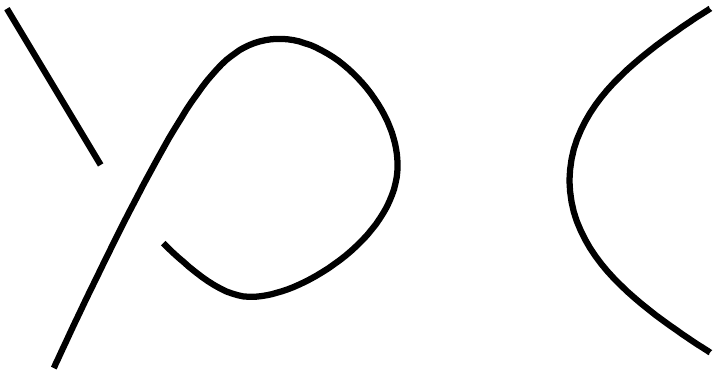}}\,,\, \raisebox{-7pt}{\includegraphics[height=0.25in]{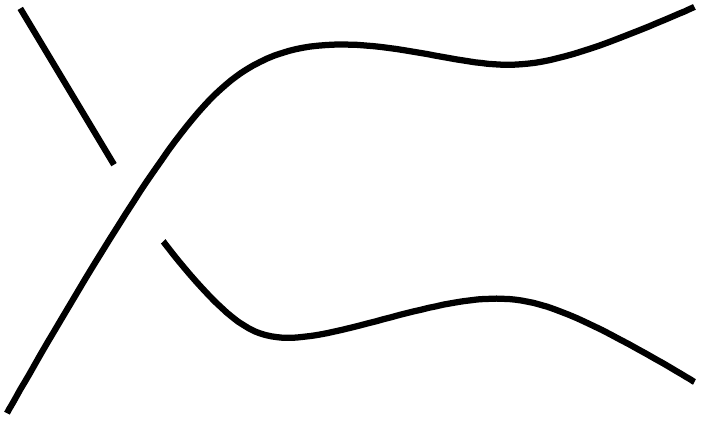}} \,\, \right\}  \]

\[ \raisebox{-10pt}{\includegraphics[height=0.35in]{N5-right}} \leadsto   \left \{ \,\, \raisebox{-7pt}{\includegraphics[height=0.3in]{poscrossing}}, \,\,\,  \raisebox{-7pt}{\includegraphics[height=0.3in]{negcrossing}},\,\,\, \raisebox{-7pt}{\includegraphics[height=0.3in]{A-smoothing}}, \,\,\,  \raisebox{-7pt}{\includegraphics[height=0.3in, angle = 90]{A-smoothing}} \,\, \right \} \]

\[ \raisebox{-10pt}{\includegraphics[height=0.35in]{N5-left2}} \leadsto \left \{\,\, \raisebox{-7pt}{\includegraphics[height=0.38in, angle = 90]{reid2-1}}, \,\,\raisebox{-7pt}{\includegraphics[height=0.27in]{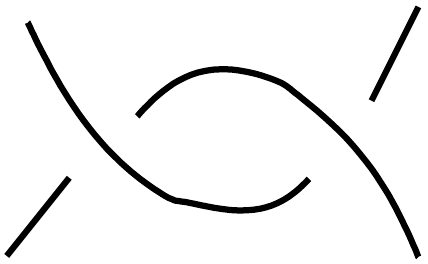}}, \raisebox{-7pt}{\includegraphics[height=0.27in]{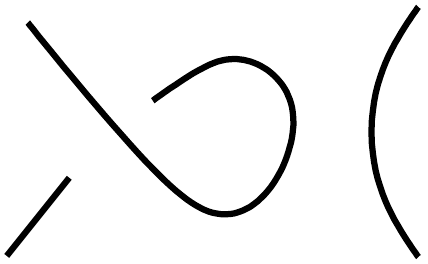}}, \,\, \raisebox{-7pt}{\includegraphics[height=0.27in]{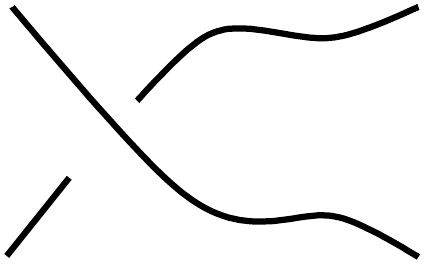}}\,\,  \right \}  \]
It is clear that, in this case, the two collections $\calC(\overline{D}_{\rho})$ and $\calC(\overline{D'}_{\rho})$ of tangle diagrams are not ambient isotopic. However,  the diagrams \raisebox{-3pt}{\includegraphics[height=0.18in]{N5-r5}} and \raisebox{-3pt}{\includegraphics[height=0.18in]{negcrossing}} differ by a $+3$-move. Similarly, the diagrams \raisebox{-3pt}{\includegraphics[height=0.18in]{N5-r2}} and  \raisebox{-3pt}{\includegraphics[height=0.18in]{poscrossing}} are related by a $3$-move.  Then it is easy to see that the collections $\calC(\overline{D}_{\rho})$ and $\calC(\overline{D'}_{\rho})$ are $3$-equivalent.

Finally, there should be no difficulty to construct the permutation $\sigma$ on the set $\{1, \dots, 4^n \}$ in the statement of the proposition.
\end{proof}

\begin{proposition}\label{prop:inv-enhanced}
Let $D_{\rho}$ be a diagram of an enhanced trivalent tangle $(G, \rho)$, and let $\overline{D}_{\rho}$ be the $4$-valent tangle diagram obtained from $D_{\rho}$ by contracting its thick edges. Then the $3$-equivalence class of the collection $\calC(\overline{D}_{\rho})$ of ordinary tangle diagrams is an ambient isotopy invariant of $(G, \rho)$, as well as an invariant of the IH-equivalence class of $(G, \rho)$.\end{proposition}

\begin{proof} 
The statement follows from Lemma~\ref{lemma} and Proposition~\ref{prop:3-move inv}.\end{proof}


 \subsection{Invariants for enhanced trivalent tangles }\label{ssec:inv trivalent tangles}
 
 An invariant $\calI$ for classical tangles is called a \textit{3-move invariant} if $\calI(T) = \calI(T')$ for any two 3-equivalent tangles $T$ and $T'$. 
 
 By Proposition~\ref{prop:inv-enhanced}, we have that if $\calI$ is a 3-move invariant for tangles, then it can be extended to an ambient isotopy invariant $\calI(G, \rho)$ of an enhanced trivalent tangle $(G, \rho)$---with associated 4-valent tangle $\overline{G}_{\rho}$---as follows:
 \[ \calI(G, \rho): = \sum_{(\overline{D}_{\rho}, f_i) \in \calC(\overline{D}_{\rho})} \calI(\overline{D}_{\rho}, f_i), \] 
 where the sum is taken over all states $(\overline{D}_{\rho}, f_i)$ of $\overline{D}_{\rho}$ and where $\overline{D}_{\rho}$ is a diagram of $\overline{G}_{\rho}$.
 Moreover, by our construction, $\calI(G, \rho)$ is invariant under the IH-move on enhanced trivalent tangles as well, and thus it yields an invariant of the IH-equivalence class of the enhanced trivalent tangle $(G, \rho)$, and equivalently, an invariant of the enhanced handlebody-tangle with spine $(G, \rho)$. 

 Furthermore, the following sum taken over all enhanced trivalent tangles $(G, \rho)$ associated to a given trivalent tangle $G$:
  \[\calI(G): = \sum_{\rho} \calI(G, \rho),\]
  yields an invariant of $G$. If the tangle has no univalent vertices, the method described here provides a recipe for constructing invariants for knotted trivalent graphs.
  
  We remark that such techniques have been used before to obtain invariants of knotted graphs. For example, the idea of using collections of tangles to obtain invariants of knotted graphs was first introduced (to the best of our knowledge) by Kauffman in~\cite{Ka}. Moreover, 3-move invariants of knots and links were used by Lee and Seo in \cite{L-S} to construct numerical invariants for knotted 4-valent graphs. Our invariant described in Section~\ref{ssec:3-move inv tangles} is closely related to that constructed in~\cite{L-S}.

\section{Some invariants for classical tangles}\label{sec:inv tangles}

Our goal now is to construct 3-move invariants for classical tangles and consequently arrive at invariants for enhanced trivalent tangles and handlebody-tangles, as explained in Section~\ref{sec:trivalent tangles}.


\subsection{The skein module $E_{m,n}$}\label{ssec:skein module}

Let $m, n$ be non-negative integers such that $m + n$ is even, and let $q$ be an indeterminate. An $(m,n)$-tangle $T$ is an embedding in $\R^2\times [0,1]$ of $\frac{1}{2}(m+n)$ arcs and a finite number of circles, with the property that the endpoints of the arcs are distinct points in $\R^2 \times \{0\}$ and $\R^2 \times \{1\}$. A tangle diagram $D$ of $T$ is a projection of $T$ onto $\R \times [0,1]$, such that the endpoints of $T$ are mapped to distinct points in the lines $\R \times \{0\}$ and $\R \times \{1\}$.   

The \textit{skein $(m,n)$-module} is the free $\Z[q, q^{-1}]$-module $E_{m, n}$ generated by equivalence classes of $(m,n)$-tangle diagrams modulo the ideal generated by elements:
 \[  \raisebox{-7pt}{\includegraphics[height=0.25in]{poscrossing}} - q\, \raisebox{-7pt}{\includegraphics[height=0.25in]{A-smoothing}} - q^{-1} \,\raisebox{-7pt}{\includegraphics[height=0.25in, angle = 90]{A-smoothing}}, \quad \mbox{and} \quad D \cup \bigcirc - \delta D, \quad \mbox{where} \,\,\delta = -q^2 - q^{-2}. \]

 Each $(m, n)$-tangle diagram $D$ represents an element of $E_{m,n}$, denoted by $\brak{D}$, and called the \textit{skein class} of $D$. There is a basis for  $E_{m, n}$ represented by \textit{flat $(m, n)$-tangle diagrams} (crossingless matchings of the $m$ + $n$ endpoints), and the coefficients of the skein class $\brak{D}$ with respect to this basis are Laurent polynomials in $q$. If $m = n = 0$, that is the tangle represented by $D$ is a link $L$, then $\brak{D}$ is the Kauffman bracket~\cite{K1} of $L$, up to a factor of $\delta = -q^2-q^{-2}$.
 
 Let $B_{m, n} = \{ e_1, e_2, \dots, e_p\}$ be the basis of $E_{m,n}$ consisting of all flat $(m,n)$-tangle diagrams. We note that $|B_{m,n}|=\binom{2k}{k}/(k+1)$ is the $k$-th Catalan number, where $k = \frac{1}{2}(m+n)$. For each $(m,n)$-tangle diagram $D$,  denote the coordinate vector of $\brak{D}$ relative to $B_{m, n}$ by  $v(D) = [x_1 \,\, x_2\,\, \dots \,\, x_p]$, where we write 
\[ \brak{D} = x_1e_1+ x_2e_2+ \dots + x_pe_p,\,\, \text{for}\,\, x_i  \in \mathbb{Z}[q, q^{-1}] .\]
We remark that $v(D)$ is a regular isotopy invariant of the tangle $T$ represented by $D$. 

If $\overline{D}$ is the \textit{mirror image} of $D$ (obtained from $D$ by replacing each over-crossing by an under-crossing and vice-versa), then $v(\overline{D}) = [\overline{x}_1 \,\, \overline{x}_2\,\, \dots \,\, \overline{x}_p]$, where $\overline{x}_i: = x_i |_{q \leftrightarrow q^{-1}}$. Specifically, $\overline{x}_i$ is obtained from $x_i$ by interchanging $q$ and $q^{-1}$.


\subsection{3-move invariants for classical tangles}\label{ssec:3-move inv tangles}

Given two $(m,n)$-tangles $T_1$ and $T_2$, denote by $T_1 \otimes T_2$ the $(2m, 2n)$-tangle obtained by placing $T_2$ to the right of $T_1$ without any intersection or linking. That is, $T_1 \otimes T_2$ is the tensor product of the morphisms represented by these tangles in the category of tangles. Denote by $cl(T_1 \otimes T_2)$ the \textit{plat closure} of $T_1 \otimes T_2$, which is a link or a knot obtained by joining  with simple arcs adjacent upper endpoints and respectively, adjacent lower endpoints of $T$. Figure~\ref{fig:plat closure} displays the plat closure of the tensor product of a $(3,3)$-tangle with its mirror image.
\begin{figure}[ht]
\[ T \otimes \overline{T}=  \raisebox{-23pt}{\includegraphics[height=0.65in]{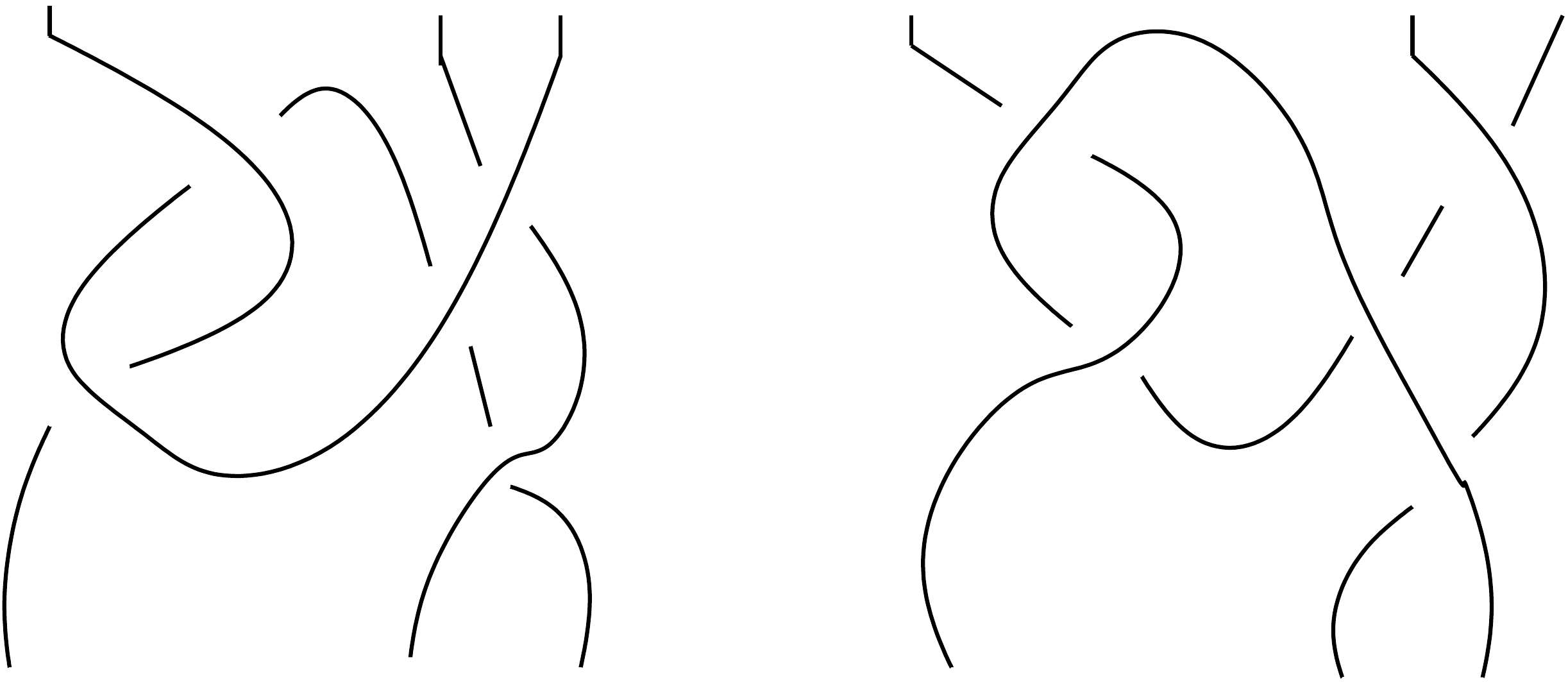}} \longrightarrow cl(T \otimes \overline{T}) = \raisebox{-23pt}{\includegraphics[height=0.65in]{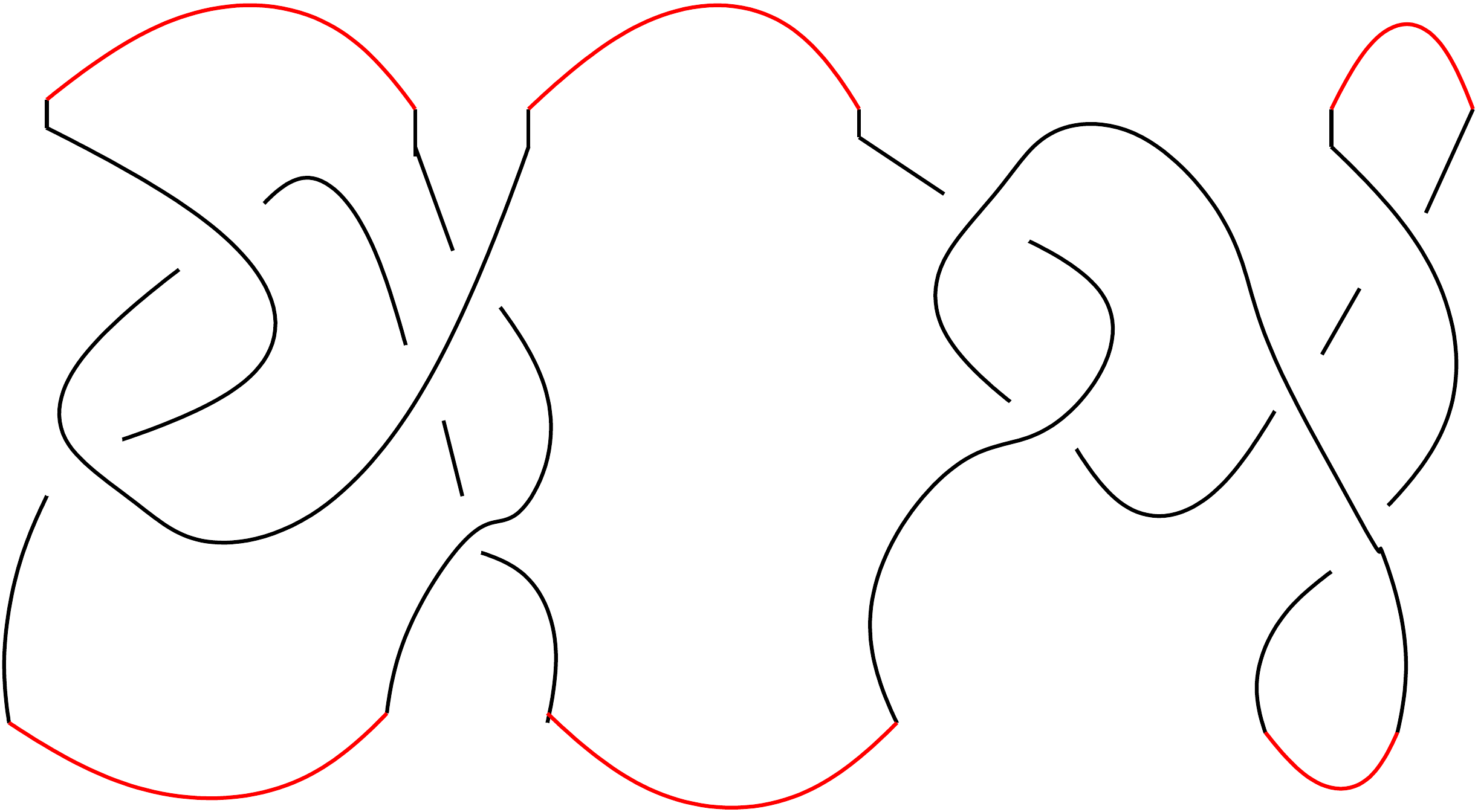}} \]
\caption{The plat closure of $T \otimes \overline{T}$}\label{fig:plat closure}
\end{figure}

Define a bilinear form $[ \,\,,\, \,] : E_{m, n} \times E_{m, n} \longrightarrow \mathbb{Z}[q, q^{-1}]$
 given by
\[ [e_i, e_j] = (-q^2 -q^{-2})^{\ell}  \quad \text{for all} \,\, 1 \leq i,j \leq p,\]
where $\ell = \#$ of closed loops in $cl(e_i \otimes e_j)$. Extend $[ \,\,,\, \,]$ by bilinearity to all elements in $E_{m, n}$.

Let $a_{ij} =  [e_i, e_j] $, and denote by $A = (a_{ij})_{1 \leq i,j \leq p}$ the matrix of $[ \,\,,\, \,]$ relative to the basis $B_{m, n}$. For any tangle diagram $D$ of an $(m, n)$-tangle, we have
\begin{eqnarray*}
[\brak{D}, \brak{\overline{D}} ] = v(D)\, A\, v(\overline{D})^t 
= [x_1 \,\, x_2\,\, \dots \,\, x_p] \,A \, \left[ \begin{array}{c} \overline{x}_1 \\ \overline{x}_2\\ \vdots \\ \overline{x}_p  \end{array} \right]  \in \mathbb{Z}[q, q^{-1}].
\end{eqnarray*}

\begin{definition} \label{def:P(D)}
Let $T$ be an $(m,n)$-tangle and let $D$ be a diagram of $T$. Denote by \[P(D): = [\brak{D}, \brak{\overline{D}} ].\]
\end{definition}

\begin{proposition} \label{prop:inv for tangles}
Let $D$ be a diagram representing an $(m, n)$-tangle $T$. Then $P(D)$ is invariant under the three Reidemeister moves.
\end{proposition}

\begin{proof} Let $T$ be an $(m, n)$-tangle, and $D$ and $D'$ be diagrams of $T$. Without loss of generality, we may assume that $D$ and $D'$ differ by exactly one of the Reidemeister moves.

Case I. Suppose that $D$ and $D'$ differ by an R2 or R3 move. Then $v(D) = v(D')$ and $v(\overline{D}) = v(\overline{D'})$, which implies that $P(D) = P(D')$. 

Case II. Suppose that $D$ and $D'$ differ by an R1 move, where $D$ has an extra twist in it. Since $\brak{\, \raisebox{-5pt}{\includegraphics[height=0.22in]{poskink}}\,} =   - q^3 \brak{\,\raisebox{-5pt}{\includegraphics[height=0.22in]{arc}}\,}$ and $ \brak{\, \overline{\raisebox{-5pt}{\includegraphics[height=0.22in]{poskink}}}\,} = \brak{\, \reflectbox{\raisebox{11pt}{\includegraphics[height=0.22in, angle = 180]{poskink}}}\,} =  -  q^{-3} \brak{\,\raisebox{-5pt}{\includegraphics[height=0.22in]{arc}}\,}$, it follows that $P(D) = (- q^3) ( -q^{-3}) P(D') = P(D')$.
\end{proof}

\begin{remark}
$P(T): = P(D)$ is the Kauffman bracket of the link/knot $cl(T\otimes \overline{T})$, up to a factor of $\delta = -q^2-q^{-2}$.
\end{remark}

We are interested in 3-move invariants for tangles, and thus, in particular, we are interested in the behavior of $\brak{\,\,\,}$ and $[ \,\,,\,\,]$ under the 3-moves for tangles. We have that:
\[ \brak{\,\raisebox{-3pt}{\includegraphics[height=0.8in, angle = 90]{+3-move}}\,} =  q^3\,  \brak{\,\raisebox{-5pt}{\includegraphics[height=0.22in, angle = 90]{A-smoothing}}\,} + (q-q^{-3} + q^{-7})\,\brak{\,\raisebox{-5pt}{\includegraphics[height=0.22in]{A-smoothing}}\,} \]
\[ \brak{\, \reflectbox{\raisebox{-3pt}{\includegraphics[height=0.8in, angle = 90]{+3-move}}} \,} =  q^{-3}\,  \brak{\,\raisebox{-5pt}{\includegraphics[height=0.22in, angle = 90]{A-smoothing}}\,} + (q^{-1}-q^{3} + q^{7})\,\brak{\,\raisebox{-5pt}{\includegraphics[height=0.22in]{A-smoothing}}\,}.   \]

Let now $q \in \mathbb{C}$ such that it is a nonzero common root of $q-q^{-3} + q^{-7}$ and $q^{-1}-q^{3} + q^{7}$. That is, let $q= e^{\frac {k \pi i}{12}}$, where $k \in \{1, 5, 7, 11, 13, 17, 19, 23\}$.

\begin{definition}
Let $T$ be an $(m,n)$-tangle represented by a diagram $D$. Denote by 
\[P(D)_k: = P(D) |_{q = e^{\frac{k \pi i}{12}}} \in \mathbb{C}\,\,\, \text{where} \,\,\, k \in \{1, 5, 7, 11, 13, 17, 19, 23\}.\]
\end{definition}

\begin{theorem} \label{main result}
Let $T$ be an $(m,n)$-tangle and let $D$ be a diagram of $T$. The complex number $P(D)_k$ is a 3-move invariant for $T$, for each $k \in \{1, 5, 7, 11, 13, 17, 19, 23\}$.
\end{theorem}

\begin{proof}
Proposition~\ref{prop:inv for tangles} implies that $P(D)_k$ is an ambient isotopy invariant for $T$, for each $k \in \{1, 5, 7, 11, 13, 17, 19, 23\}$. Let $D'$ be a diagram that is identical to $D$ except in a neighborhood where it differs from $D$ by a $3$-move. Then $P(D')_k = q^3  q^{-3} P(D)_k = P(D)_k$, for each $k \in \{1, 5, 7, 11, 13, 17, 19, 23\}$.  Therefore $P(D)_k$ is invariant under the $3$-moves, as well. 
\end{proof}

We remark that in the case of links, the numerical invariant $P(D)_k$ of Theorem~\ref{main result} is the invariant appearing in Theorem 4.4 of~\cite{L-S}.
\medskip


\textbf{Conclusions and final comments.} Using the method described in Section~\ref{sec:trivalent tangles} with the generic $\calI(T)$ (see Section~\ref{ssec:inv trivalent tangles}) replaced by the 3-move  tangle invariant $P(D)_k$ obtained in Section~\ref{ssec:3-move inv tangles}, we arrive, for each $k \in \{1, 5, 7, 11, 13, 17, 19, 23\}$, at a numerical invariant of the IH-equivalence classes of enhanced trivalent tangles, and moreover, at an invariant of trivalent tangles. For each $k$ as above, this yields a numerical invariant of enhanced handlebody-tangles. Our results hold for knotted trivalent graphs and enhanced handlebody-links, as well. We remark that the numerical invariants obtained here are not independent: for all $k \in \{1, 5, 7, 11, 13, 17, 19, 23\}$, $q = e^{\frac{k \pi i}{12}}$ are primitive $24$-th roots of unity, and the corresponding invariants can be obtained from one another by Galois group actions. (The author would like to thank the referee for pointing this out.)

The 3-move invariants for classical tangles constructed here were defined using elementary concepts from linear algebra. As noted in Section~\ref{ssec:3-move inv tangles}, these 3-move invariants for tangles are equivalent to certain evaluations of the Kauffman bracket polynomial of a link obtained in a specific way from the original tangle. The reader may want to compare this with Przytycki's~\cite{P} analysis of how the 3-moves influence the Jones polynomial~\cite{Jo}. In~\cite{P} it was also observed that tricoloring and $F(1, -1)$ are 3-move invariants of links, where $F$ is the Kauffman two-variable polynomial~\cite{K2}. 

We remark that Montesinos and Nakanishi conjectured that every link can be reduced to a trivial link by a sequence of 3-moves. Dabkowski and Przytycki~\cite{DP} found obstructions to this conjecture: they showed that the Borromean rings are not 3-equivalent to a trivial link. They also found a braid on three strands and with 20 crossings whose closure cannot be reduced by 3-moves to a diagram of a trivial link.



\begin{thebibliography}{99}
\bibitem{DP} M. K. Dabkowski, J. H. Przytycki, {\em Burnside obstructions to the Montesinos-Nakanishi
3-move conjecture}, Geom. Topol. \textbf{6} (2002), 355-360.

\bibitem{I1} A. Ishii, {\em Moves and invariants for knotted handlebodies}, Algebr. Geom. Topol. \textbf{8} (2008), 1403-1418.

\bibitem{I2} K. Ishihara, A. Ishii, {\em An operator invariant for handlebody-knots}, Fund. Math. \textbf{217} (2012), 233-247.

\bibitem{Jo} V. F. R. Jones, \textit{A polynomial invariant for links via von Neumann algebras}, Bull. Amer. Math. Soc. \textbf{129} (1985), 103--112.

\bibitem{K1} L. H. Kauffman, \textit{State models and the Jones polynomial}, Topology \textbf{26}, No. 3 (1987), 395-407.

\bibitem {Ka} L. H.  Kauffman, {\em Invariants of graphs in three-space}, Trans. Amer. Math. Soc. \textbf{311} (1989) 697-710. 

\bibitem{K2} L.H. Kauffman, {\em An invariant of regular isotopy}, Trans. Amer. Math. Soc. \textbf{318}, No. 2 (1990), 417-471.

\bibitem{L-S} S. Y Lee, M. Seo, {\em The 3-move and knotted 4-valent graphs in 3-space}, Osaka J. Math. \textbf{41} (2004), 119-130.

\bibitem{P} J. H. Przytycki, {\em 3-coloring and other elementary invariants of knots}, Banach Center Publications, \textbf{42} (1998), 275-295; see also arXiv:math.GT/0608172.
\end{thebibliography}
\end{document}